\documentclass{article}
\usepackage[utf8]{inputenc}
\usepackage{blindtext}
\usepackage{xcolor}
 \usepackage{epsfig,float}
 \usepackage{wrapfig,lipsum}
\usepackage{amsthm}
\usepackage{graphicx,amssymb,amsmath,latexsym}
\usepackage{caption}
\usepackage{subcaption}

\title{\textbf{Singular twisted links and singular twisted virtual braids}}
\author{{ \small KOMAL NEGI}, {\small MADETI PRABHAKAR}}

\theoremstyle{plain}
\newtheorem{theorem}{Theorem}[section]
\newtheorem{lemma}[theorem]{Lemma}

\newtheorem{corollary}[theorem]{Corollary}

\theoremstyle{definition}
\newtheorem{definition}[theorem]{Definition}
\newtheorem{remark}[theorem]{Remark}
\newtheorem{Example}[theorem]{Example}
\date{}

\begin{document}

\maketitle

\begin{abstract}
The concepts of twisted knot theory and singular knot theory inspire the introduction of singular twisted knot theory. This study showcases similar findings for singular twisted links, including the Alexander theorem and the Markov theorem derived from knot theory. Moreover, in this paper we define singular twisted virtual braids and their monoid structure. Additionally, we provide both a monoid and a reduced monoid presentation for singular twisted virtual braids.\\



\end{abstract}

\section{Introduction}

S. Kamada~\cite{sk} proved that every oriented virtual link can be depicted as the closure of a virtual braid. 
Furthermore, it is established that two virtual braids have isotopic closures if and only if they are connected by virtual braid isotopy and a finite series of Markov moves, a widely recognized result. 
Similar theorems have been established by Kauffman and Lambropoulou~\cite{kl,KS} through the use of $L_v$-equivalence. 
Moreover, C. Caprau et al.~\cite{cpas} extended the $L_v$-move approach to virtual singular braids, deriving the analogous $L_v$-move for virtual singular braids.

Recently, Kamada et al.~\cite{NPK} established Alexander and Markov's theorems for twisted virtual braids and twisted links~\cite{MO,NKSK}. 
Furthermore, they verified that the collection of twisted virtual braids with $n$ strands forms a group~\cite{VTKM}. It will be interesting to see similar kind of results for singular twisted virtual braids.

This paper focuses on oriented singular twisted links and establishes Alexander and Markov-type theorems for this category of links. These theorems play a important role in understanding the structure of singular twisted knots and links. First, we define the singular twisted virtual braid monoid through generators and relations. This definition unveils that the singular twisted virtual braid monoid with $n$ strands extends from the virtual singular braid monoid with $n$ strands by a group, which is isomorphic to $\mathbb{Z}_2^n$.

Several braiding techniques can be employed to demonstrate the extension of the Alexander theorem to the category of singular twisted virtual braids. 
In our context, we adopt the braiding algorithm outlined in~\cite{cpas}.

This article is organized as follows,
in Section 2, we provide some preliminaries of singular twisted links and singular twisted virtual  braids. Also, we stated the Alexander theorem for singular twisted links.
In Section 3, we outline $L_v$-moves and present the Markov-type theorems for singular twisted links and singular twisted virtual braids. Moreover, we have established that number of Markov moves or $L_v$-moves can be reduced for singular twisted virtual braids.
In Section 4, we provide a reduced presentation of the singular twisted virtual braid monoid.

\section{Preliminaries}
In this section we define singular twisted links in topological and combinatorial way. We also define singular twisted virtual braids. Moreover, we give Alexander theorem for singular twisted links.
\subsection{Singular twisted link}
\begin{definition}
    A \textit{singular twisted link} is an embedding of a 4-valent graph $G$ in $S_g \times I$, where $S_g$ is a surface of the genus $g$ and $I$ is the unit interval such that the thickened surface $S_g \times I$ is orientable.
\end{definition}
\begin{remark}
    A \textit{singular twisted knot} is an embedding of a one component 4-valent graph $G$ in $S_g \times I$.
\end{remark}

There is also a combinatorial way to define singular twisted links. For that we need to define singular twisted link diagrams. 
\begin{definition}
    A singular twisted link diagram is a decorated immersion of finitely many disjoint copies of $S^1$ into $\mathbb{R}^2$, with bars and finitely many transverse double points such as classical crossings, virtual crossings, and singular crossings. 
\end{definition}
Two singular twisted link diagrams are said to be equivalent(or ambient isotopic) if they are related by a sequence of Reidemeister moves shown in Figure~\ref{rts}.
\begin{figure}[h]
	\centering
\includegraphics[width=10cm,height=8cm]{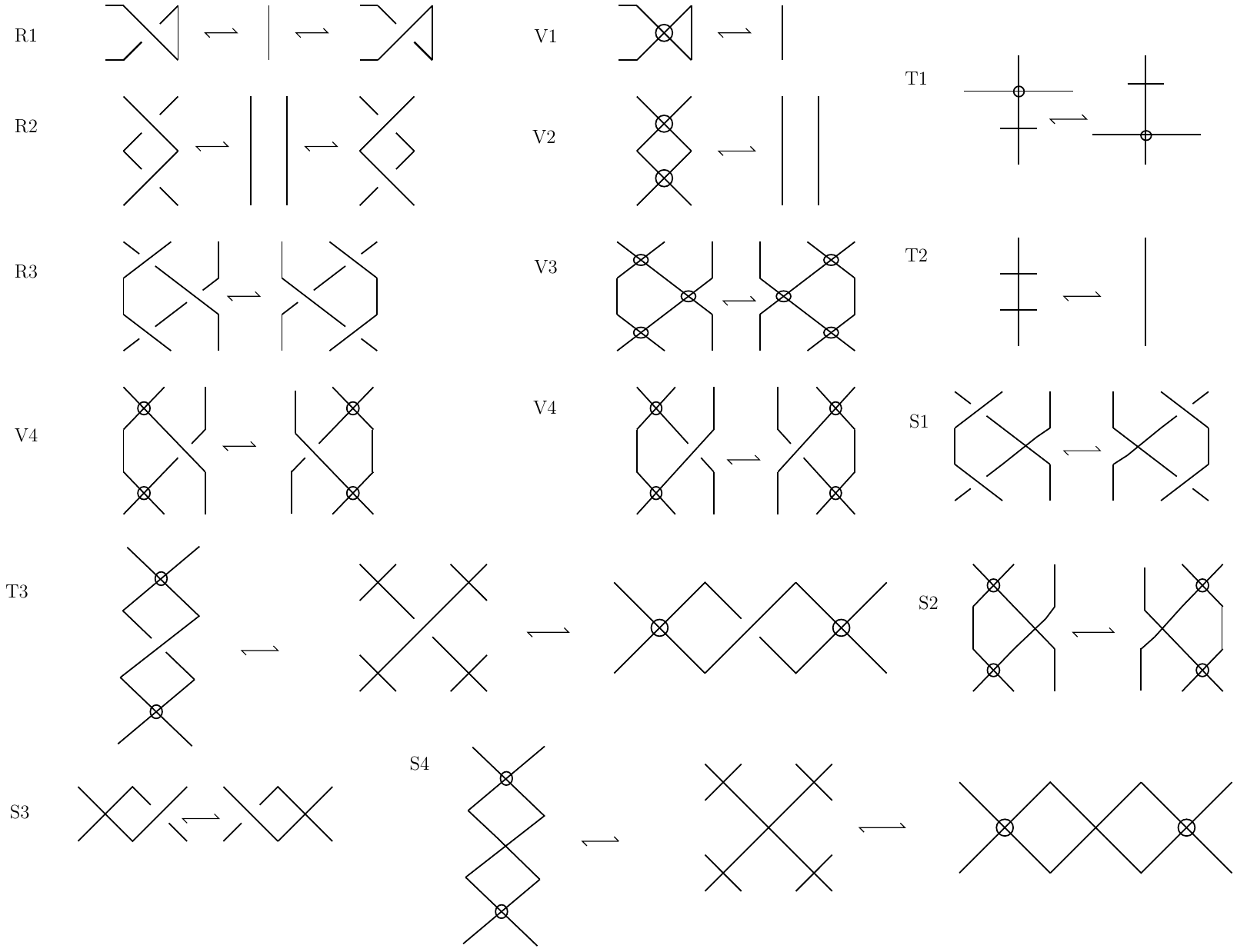}	
  \caption{Reidemeister moves for singular twisted link diagrams}
 \label{rts}
   \end{figure}

\begin{definition}
    A singular twisted link is the equivalence class of a singular twisted link diagram.
\end{definition} 
\begin{Example}
    Consider the two singular twisted link diagrams $D_1$ and $D_2$ as shown in Figure~\ref{stkd}. $D_2$ can be obtained from $D_1$ by sequence of Reidemeister moves S4, S3, and T2, respectively.
    \begin{figure}[h]
	\centering
\includegraphics[width=10cm,height=3cm]{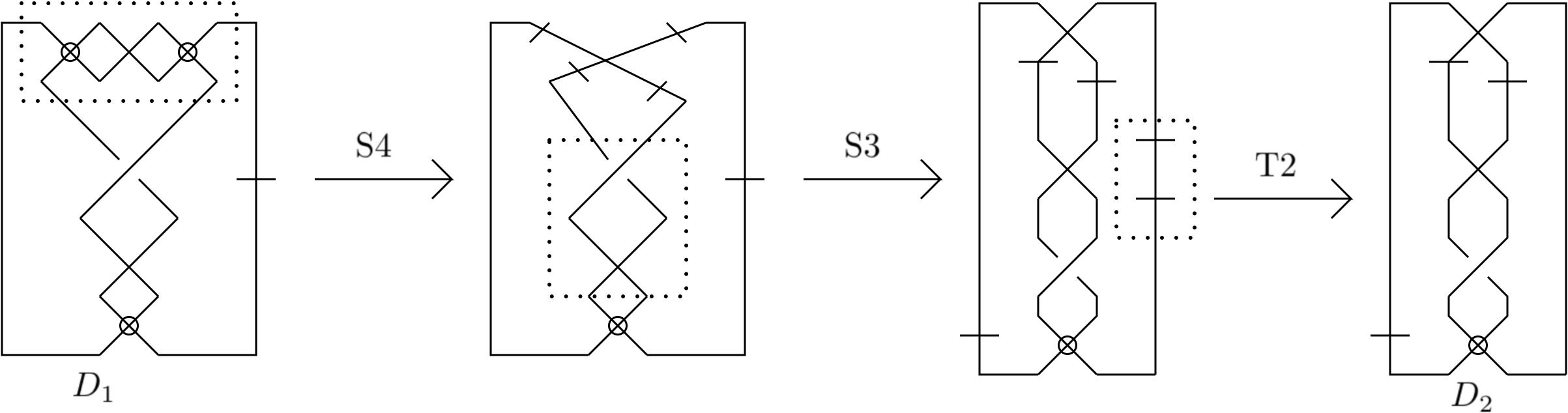}	
  \caption{Singular twisted link diagrams}
 \label{stkd}
   \end{figure}
\end{Example}

As we know, classical knots have a nice connection with braids. A similar connection exists between virtual knots and virtual braids. Therefore, it is natural to ask: Does there exist any braid-like structure in this theory? In the next subsection, we define singular twisted virtual braids, and we demonstrate their connection to singular twisted links.

\subsection{Singular twisted virtual braid}

\begin{definition}
A {\it singular twisted virtual braid diagram} on $n$ strands (or of degree $n$) 
is a union of $n$ smooth or polygonal curves, which are called {\it strands}, in $\mathbb{R}^2$ connecting points $(i,1)$ with points $(q_i,0)$ $(i=1, \dots, n)$, where $(q_1, \ldots, q_n)$ is a permutation of the numbers $(1, \ldots, n)$, such that these curves are monotonic with respect to the second coordinate and intersections of the curves are transverse double points equipped with information as a positive/negative/virtual/singular crossing and curves may have {\it bars} by which we mean short arcs intersecting the curves transversely.
See Figure~\ref{ext}, where the five crossings are negative, positive, virtual, singular and positive from the top. 
\end{definition}
\begin{figure}[h]
  \centering
    \includegraphics[width=2cm,height=3cm]{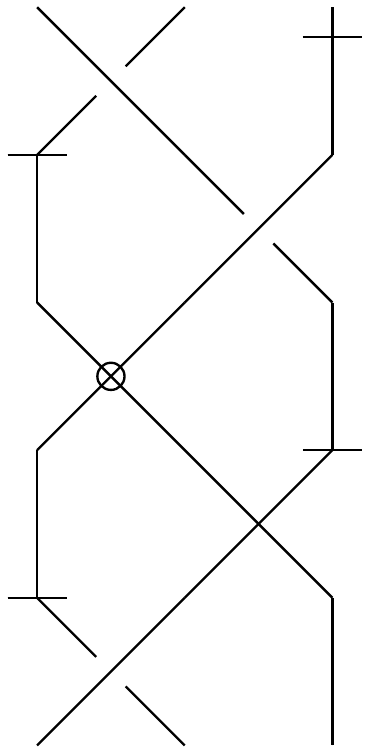}
        \caption{A singular twisted virtual braid on 3 strands}
        \label{ext}
        \end{figure}

\begin{definition}
Two singular twisted virtual braid diagrams $b$ and $b'$ of degree $n$ are {\it equivalent} if they are related by classical, virtual, twisted and singular braid moves shown in Figure~\ref{bmoves},\ref{vbmoves},\ref{moves}, and \ref{smoves}, respectively. 
\end{definition}

\begin{figure}[h]
  \centering
    \includegraphics[width=8cm,height=3.5cm]{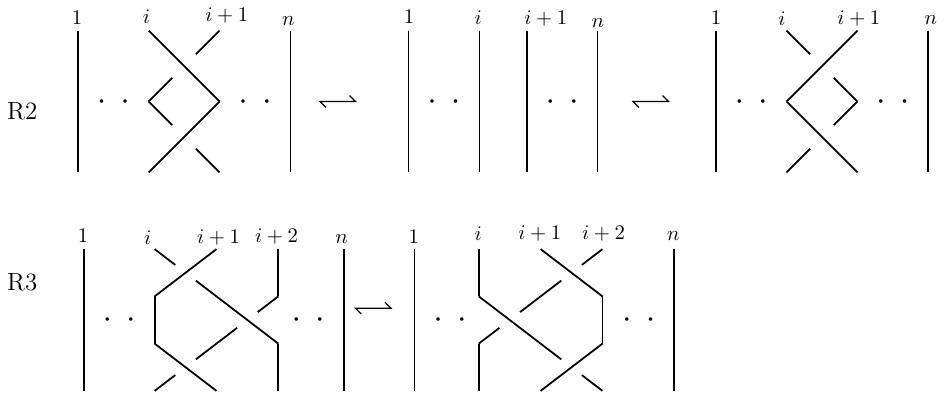}
        \caption{Classical braid moves}
        \label{bmoves}
        \end{figure}  
        
\begin{figure}[h]
  \centering
    \includegraphics[width=10cm,height=4cm]{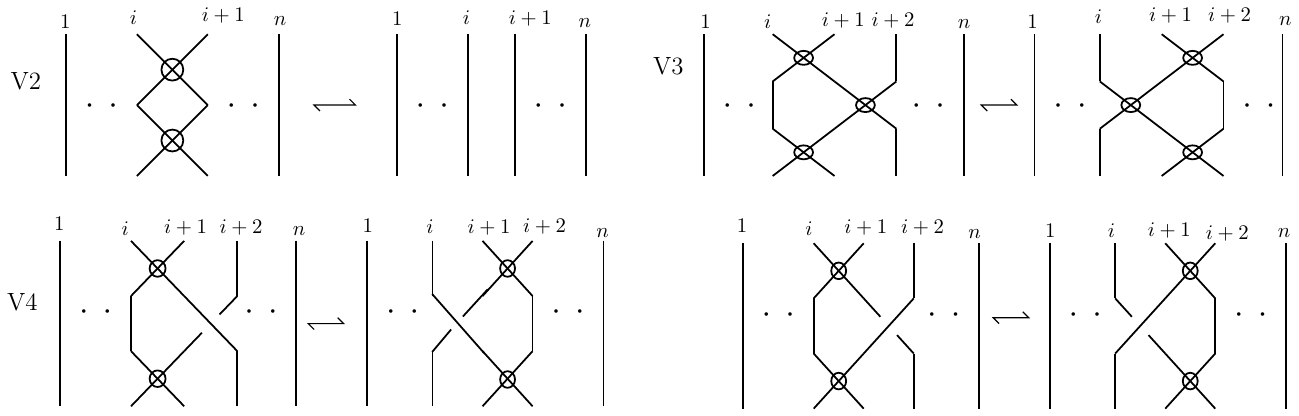}
        \caption{Virtual braid moves}
        \label{vbmoves}
        \end{figure}  
        
 \begin{figure}[h]
  \centering
    \includegraphics[width=9cm,height=4cm]{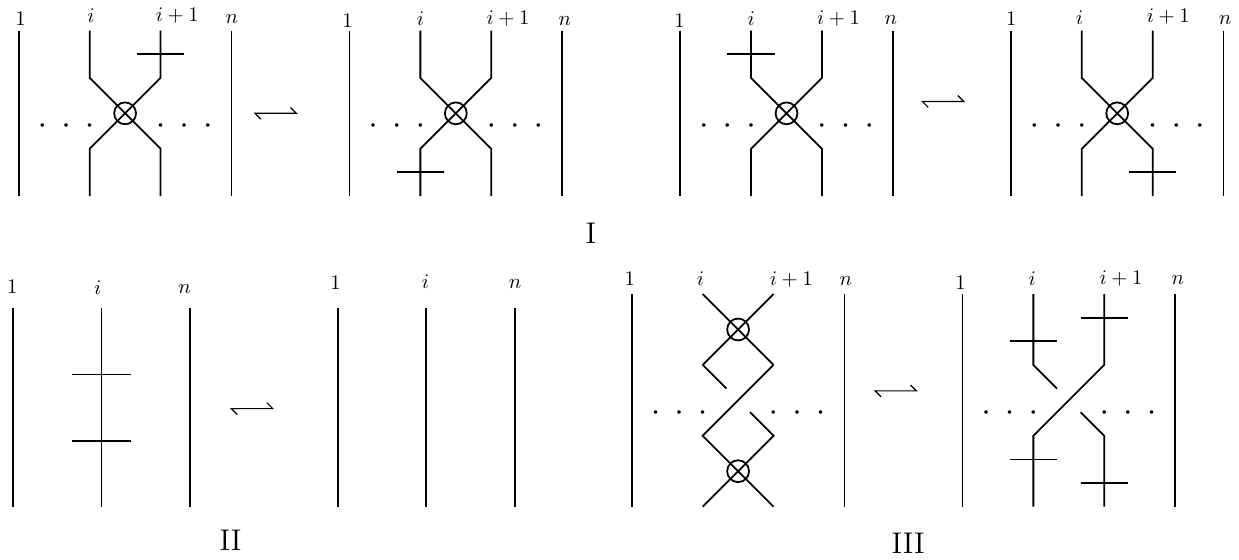}
        \caption{Twisted braid moves}
        \label{moves}
        \end{figure} 

 \begin{figure}[h]
  \centering
    \includegraphics[width=9cm,height=4.5cm]{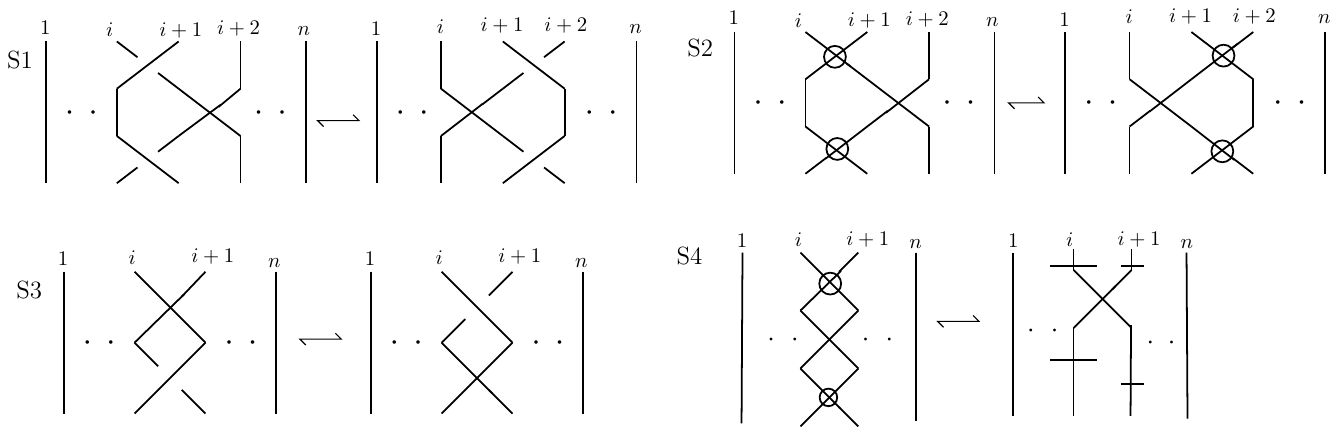}
        \caption{Singular braid moves}
        \label{smoves}
        \end{figure}

\begin{definition}
     A {\it singular twisted virtual braid} is an equivalence class of singular twisted virtual braid diagrams. 
\end{definition}
Similar to the case of singular braids and virtual singular braids, the set of singular twisted virtual braids forms a monoid, where the product is defined by the concatenation similar to the braid group.
The set of isotopy classes of singular twisted virtual braid on $n$ strands forms a monoid, which we denoted by $STVB_n$.  Define $STVB_\infty= \bigcup\limits_{n=1}^{\infty}STVB_n$.
\begin{figure}[h]
  \centering
    \includegraphics[width=12cm,height=2cm]{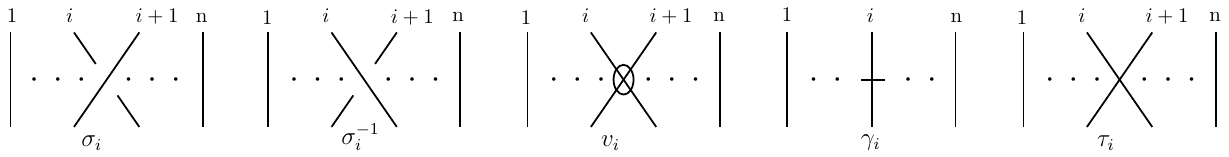}
        \caption{Generators of the monoid of singular twisted virtual braids}
        \label{gen}
        \end{figure}

\begin{theorem}\label{thm:StandardPresentation2}
The monoid $STVB_n$ is generated by standard generators, 
$\sigma_i^{\pm 1}$, $v_i$ $(i=1, \dots, n-1)$, $\gamma_i$ $(i=1, \dots, n)$, and  $\tau_i$  $(i=1, \dots, n-1)$ satisfy
the following relations, where $e$ denotes the identity element:
    \begin{align}
       \sigma_i \sigma_j & = \sigma_j \sigma_i  & \text{ for } & |i-j| > 1; \label{rel-height-ss}\\
        \sigma_i \sigma_{i+1} \sigma_i & = \sigma_{i+1} \sigma_i \sigma_{i+1} & \text{ for } & i=1,\ldots, n-2; \label{rel-sss}\\
        v_i^2 & = e  & \text{ for } & i=1,\ldots, n-1; \label{rel-inverse-v}\\
        v_i v_j & = v_j v_i & \text{ for } & |i-j| > 1 ; \label{rel-height-vv}\\
        v_i v_{i+1} v_i & = v_{i+1} v_i v_{i+1} & \text{ for } & i=1,\ldots, n-2; \label{rel-vvv}\\
        \sigma_i v_j & = v_j \sigma_i &  \text{ for } & |i-j| >1  ; \label{rel-height-sv}\\
        v_i \sigma_{i+1} v_i & = v_{i+1} \sigma_i v_{i+1} & \text{ for } & i=1,\ldots, n-2; \label{rel-vsv}\\
        \gamma_i^2 & = e & \text{ for } & i=1,\ldots, n; \label{rel-inverse-b}\\  
        \gamma_i \gamma_j & = \gamma_j \gamma_i   & \text{ for } & i,j=1,\ldots, n; \label{rel-height-bb} \\
        \gamma_j v_i & = v_i \gamma_j & \text{ for } & j\neq i, i+1; \label{rel-height-bv}\\
        \sigma_i\gamma_j & = \gamma_j\sigma_i & \text{ for } & j\neq i, i+1; \label{rel-height-sb}\\
        \gamma_{i+1} v_i & = v_{i} \gamma_i & \text{ for } & i=1,\ldots, n-1; \label{rel-bv} \\
        v_{i} \sigma_i v_{i} & = \gamma_{i+1} \gamma_i \sigma_{i} \gamma_i \gamma_{i+1} & \text{ for } &  i=1,\ldots, n-1; \label{rel-twist-III}\\
        \sigma_i \sigma_i^{-1} & = e  & \text{ for } & i=1,\ldots, n-1; \label{rel-height-ss1}\\
        \tau_i \tau_j & = \tau_j \tau_i & \text{ for } & |i-j| > 1 ; \label{rel-height-vv2}\\
        \sigma_i \tau_j & = \tau_j \sigma_i &  \text{ for } & |i-j| >1  ; \label{rel-height-sv1}\\
        \sigma_i \tau_{i} & = \tau_i \sigma_i &  \text{ for } &  i=1,\ldots, n; \label{rel-height-sv2}\\
        \sigma_i \sigma_{i+1} \tau_i & = \tau_{i+1} \sigma_i \sigma_{i+1} & \text{ for } & i=1,\ldots, n-2; \label{rel-sss1}\\
        \tau_i v_j & = v_j \tau_i & \text{ for } & |i-j| > 1 ; \label{rel-height-vv1}\\
        v_i \tau_{i+1} v_i & = v_{i+1} \tau_i v_{i+1} & \text{ for } & i=1,\ldots, n-2; \label{rel-vvv1}\\
        \tau_i\gamma_j & = \gamma_j\tau_i & \text{ for } & j\neq i, i+1; \label{1rel-height-sb}\\
       v_{i} \tau_i v_{i} & = \gamma_{i+1} \gamma_i \tau_{i} \gamma_i \gamma_{i+1} & \text{ for } &  i=1,\ldots, n-1. \label{1rel-twist-III}
    \end{align}
\end{theorem}
\begin{remark}
    If we consider generators $\sigma_i$, $v_i$ $(i=1, \dots, n-1)$, and $\gamma_i$ $(i=1, \dots, n)$ with the relations from (\ref{rel-height-ss}) to (\ref{rel-twist-III}), then this set forms a group called the twisted virtual braid group $TVB_n$~\cite{VTKM}.
\end{remark}
As of now, it is evident that a singular twisted virtual braid can yield a singular twisted link through the closure operation. The query arises: can we generate a singular twisted virtual braid from a singular twisted link? The answer is yes, there exists an algorithm for this process.

The method to convert any diagram of a singular twisted link into the form of a singular twisted virtual braid's closure is same as the braiding algorithm outlined in~\cite{cpas}, the only difference is the incorporation of bars on the strands. The given algorithm helps to establish a theorem concerning singular twisted links, analogous to the Alexander theorem for twisted virtual braids and twisted links~\cite{NPK}.

\begin{theorem}\label{Alexender type}(Alexander type theorem for singular twisted links)
    Every oriented singular twisted link can be represented as the closure of singular twisted virtual braid.
\end{theorem}
\begin{corollary}
    Every oriented twisted link can be represented as the closure of twisted virtual braid.
\end{corollary}

We have observed that if two singular twisted virtual braids are equivalent, their corresponding closures are equivalent as singular twisted links. However, the converse is not necessarily true. It is possible for two non-equivalent singular twisted virtual braids to have equivalent closures. To establish this converse, we can introduce a new notion of equivalence between the singular twisted virtual braids, known as $L_v$-equivalence or Markov equivalence.
\section{$L_v$-moves and Markov-type theorems for singular twisted virtual braids}

In this section we show a theorem on braid presentation of singular twisted links which is analogous to the
Markov theorem for twisted links. For this, we present the $L_v$-moves applicable to singular twisted virtual braids, wheras the literature already includes $L_v$ moves for virtual braids and singular virtual braids. Furthermore, we have demonstrated the capability to decrease the number of $L_v$-moves required for singular twisted virtual braids. 

All types of $L_v$-moves for singular twisted virtual braids are stated below:

A \textit{$L_v$-move of type 1}, denoted by \textit{ST1-move} contains two types of moves, the virtual and real $L_v$-move (denoted as VST1-move and RST1-move, respectively).  These moves are illustrated in Figure~\ref{Lv1move}. 
    \begin{figure}[h]
  \centering
    \includegraphics[width=10cm,height=2.5cm]{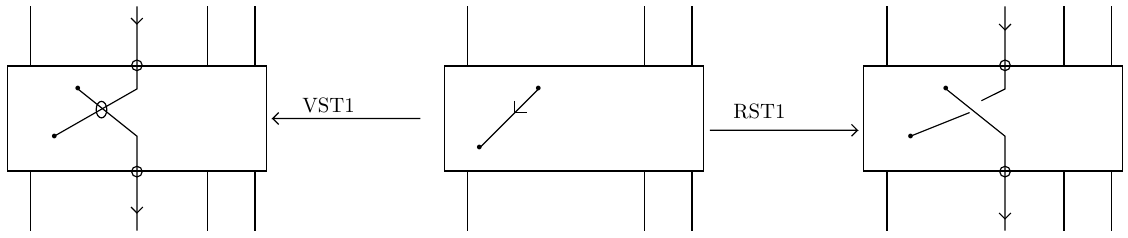}
        \caption{ST1-move: Virtual or real $L_v$-move}
        \label{Lv1move}
        \end{figure}     
        
  A \textit{$L_v$-move of type 2}, denoted by \textit{ ST2-move} is a under threaded $L_v$-move shown in Figure~\ref{Lv3move}. It contains two types of moves, the left ST2-move and the right ST2-move (denoted as LST2-move and RST2-move, respectively).  
     \begin{figure}[h]
  \centering
    \includegraphics[width=10cm,height=2.5cm]{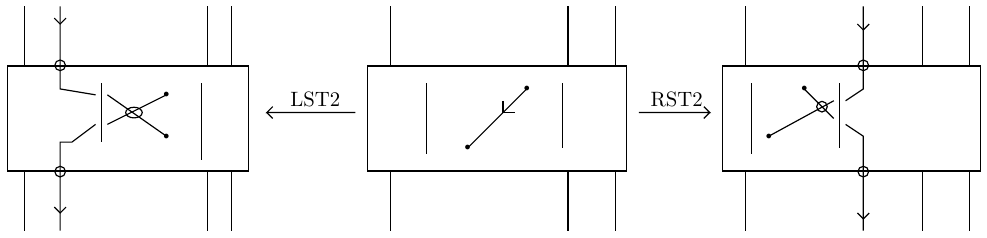}
        \caption{ST2-move: Left and Right under threaded $L_v$-move}
        \label{Lv3move}
        \end{figure} 
        
 A \textit{$L_v$-move of type 3}, denoted by \textit{ ST3-move} is a $rs$-threaded $L_v$-move. There are two types of them denoted as LST3 and RST3 shown in Figure~\ref{Lv4move}.
     \begin{figure}[h]
  \centering
    \includegraphics[width=12cm,height=2.5cm]{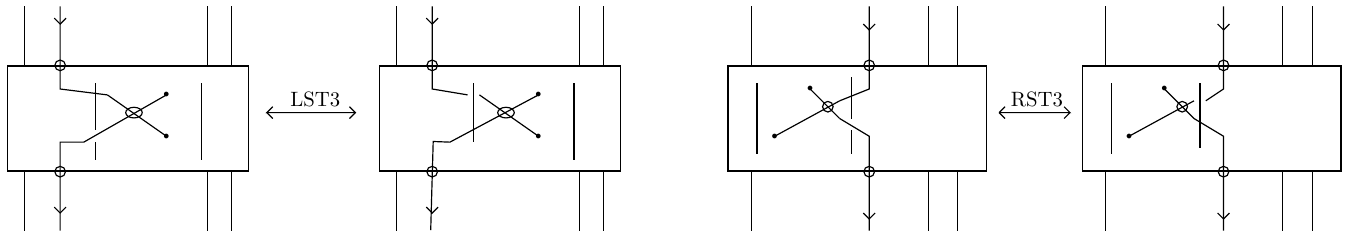}
        \caption{ST3-move: Left and Right $rs$-threaded $L_v$-move}
        \label{Lv4move}
        \end{figure} 
\begin{figure*}[h]
    \centering
    \begin{subfigure}[h]{0.5\textwidth}
        \centering
        \includegraphics[height=0.85 in]{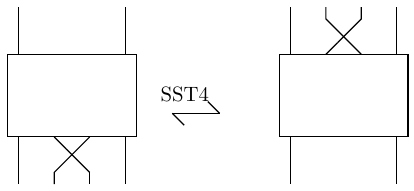}
        \caption{Singular conjugation move}
         \label{Lvcmove1}
    \end{subfigure}%
    ~ 
    \begin{subfigure}[h]{0.5\textwidth}
        \centering
        \includegraphics[width=6cm,height=2.2cm]{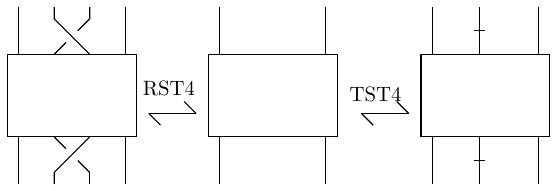}
        \caption{Real and Twist conjugation move.}
         \label{Lvcmove}
    \end{subfigure}
    \caption{Different types of ST4-move}
\end{figure*}

A \textit{$L_v$-move of type 4}, denoted by \textit{ ST4-move} is a conjugation move. There are three types of conjugation moves: Singular, real, and twist conjugation (denoted as SST4, RST4 and TST4-move, respectively) as shown in Figure~\ref{Lvcmove1} and Figure~\ref{Lvcmove}.  

\begin{remark}
In~\cite{KS}, it has been shown that virtual conjugation, basic $L_v$-moves, left $rL_v$- and $vL_v$-moves, over-threaded $L_v$-moves, and multi-threaded $L_v$-moves can be derived from the $L_v$-equivalence for virtual braids(it encompasses real conjugation, right $rL_v$- and $vL_v$-moves, left and right under-threaded $L_v$-moves, and virtual braid isotopy). Consequently, these moves can also be derived from the ST1-ST4 $L_v$-moves, thus we do not require to write them in our $L_v$-move Markov-type theorem for singular twisted virtual braids.
\end{remark}
Now we are ready to define new notion of equivalence for the singular twisted virtual braids.

\begin{definition}
    Two twisted virtual braids are said to be twisted $L_v$-equivalent if they differ by twisted virtual braid isotopy and a finite sequence of $L_v$ moves \textit{ST1-ST2, RST4}, and \textit{TST4}.
\end{definition}
\begin{definition}\label{def1}
    Two singular twisted virtual braids are said to be $L_v$-equivalent if they differ by singular twisted virtual braid isotopy and a finite sequence of $L_v$ moves \textit{ST1-ST4}.
\end{definition}

\begin{theorem}
    Two singular twisted virtual braids have isotopic closures if and only if they are $L_v$-equivalent.
\end{theorem}
\begin{proof}
If two singular twisted virtual braids are $L_v$-equivalent then they have isotopic closures.
    \begin{figure}[h]
  \centering
    \includegraphics[width=12cm,height=8cm]{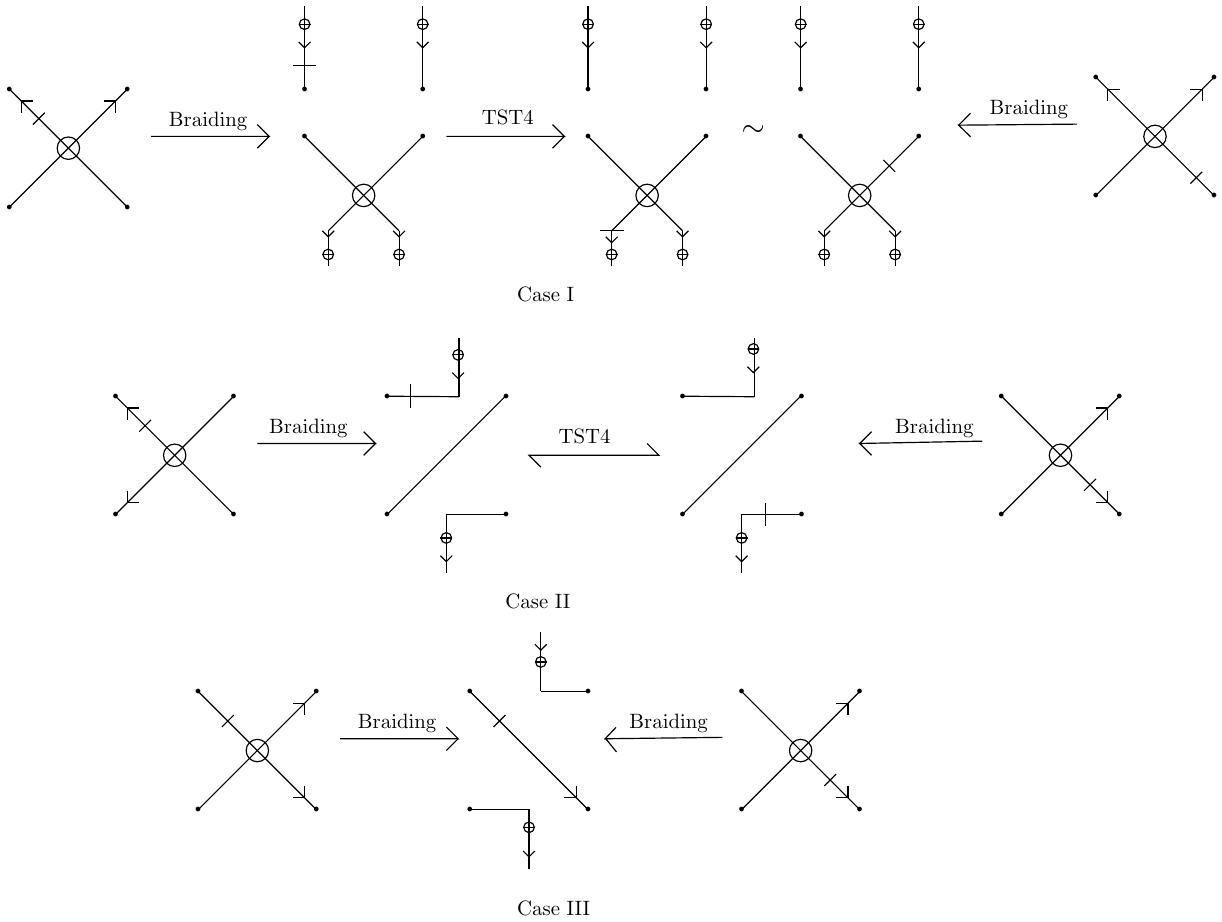}
        \caption{Different cases of T1 move}
        \label{st1move}
        \end{figure} 
For the converse, we aim to demonstrate that two singular twisted virtual braids with isotopic closures are related by $L_v$-equivalence. To achieve this, our task is to validate that singular twisted link diagrams, which differ due to Reidemeister moves, correspond to closures of singular twisted virtual braids that are $L_v$-equivalent. Starting with Reidemeister moves R1, R2, R3, V1, V2, V3, V4 moves, these moves has been checked in~\cite{KS} and S1, S2, S3 moves has been checked in~\cite{cpas}, hence it is enough to validate T1, T2, T3, and S4 moves from Reidemeister moves for singular twisted links. These moves need to be considered with any given orientation of the strands. We explore all instances of each isotopy move. If both the strands are in downward direction, then they are related by braid isotopy and braid moves. If at least one strand is oriented upward, then we proceed as follow:
 
     \begin{figure}[h]
  \centering
    \includegraphics[width=8cm,height=2cm]{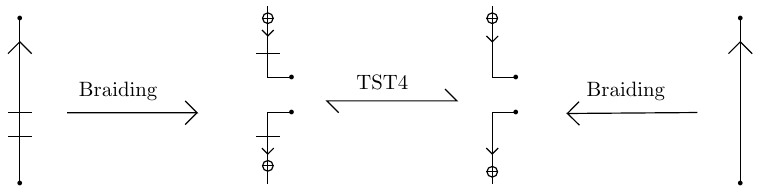}
        \caption{Braiding of T2 move}
        \label{st2move}
        \end{figure}  
   For T1 move and T2 move, all cases are illustrated pictorially in Figure~\ref{st1move} and Figure~\ref{st2move}, respectively.
   
     For T3 move, there are four cases. Case I, when both strands are in upward direction, we apply the braiding algorithm on the diagrams on both sides of the move (followed by braiding isotopy), the resulting diagrams are related by ST4-moves as shown in Figure~\ref{st3move}. Case II, when one strand is upward and one strand is downward as shown in Figure~\ref{st3bmove}, we apply the braiding algorithm on the diagrams on both sides of the move, the resulting diagrams are related by ST4-move and ST1-move.
     Case III is a different form of Case I as shown in Figure~\ref{st33move}.
    Case IV is a different form of Case II illustrated in  Figure~\ref{st34move}. 
    S4 move is same as T3 move.
      \begin{figure}[h]
  \centering
    \includegraphics[width=11cm,height=8cm]{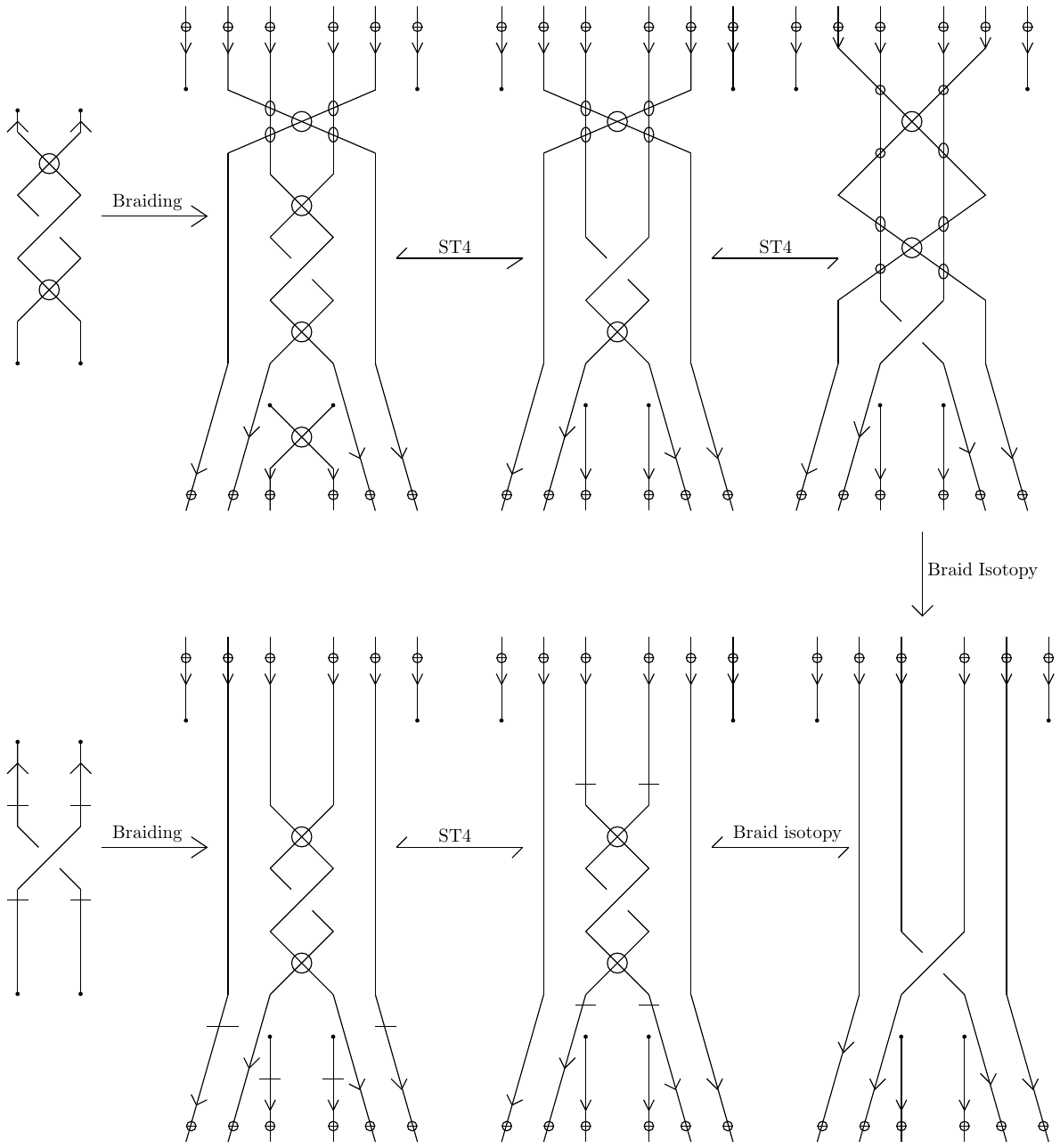}
        \caption{Case I of T3 move}
        \label{st3move}
        \end{figure} 
        \begin{figure}[h]
  \centering
    \includegraphics[width=11cm,height=8cm]{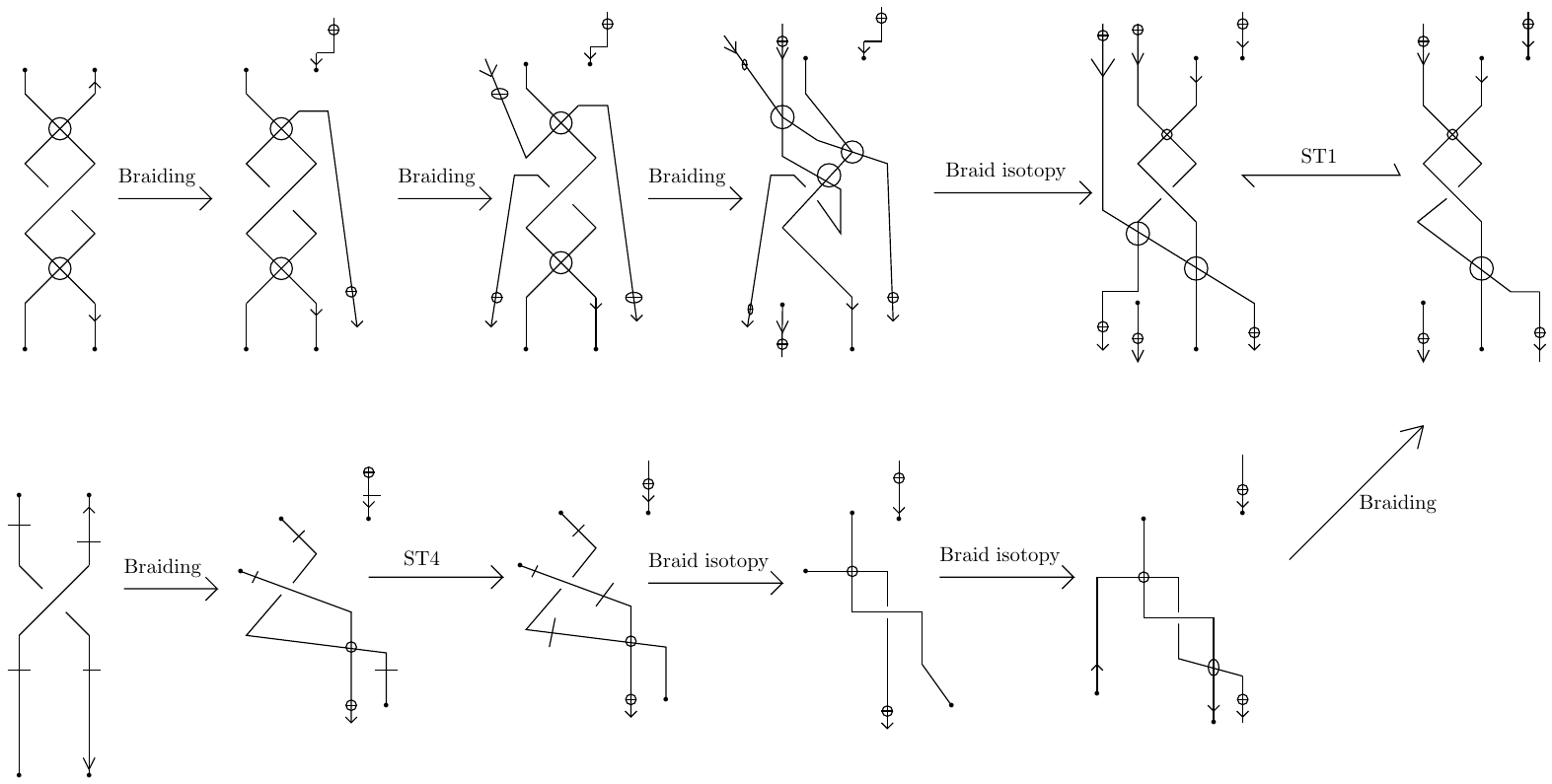}
        \caption{Case II of T3 move}
        \label{st3bmove}
        \end{figure} 
          \begin{figure}[h]
  \centering
    \includegraphics[width=9cm,height=5.5cm]{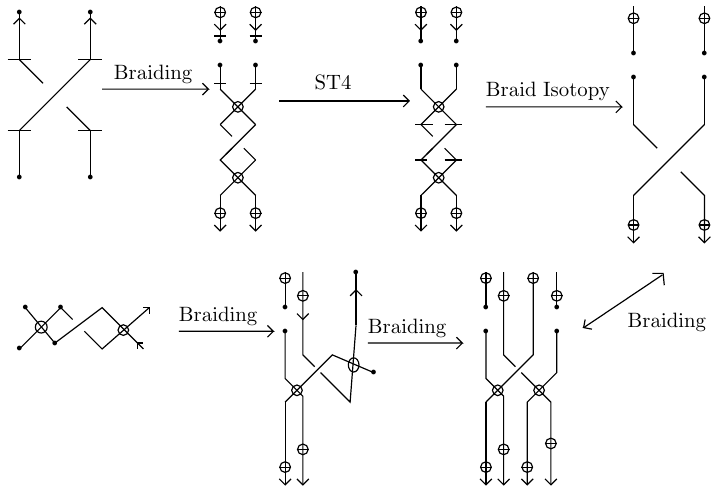}
        \caption{Case III of T3 move}
        \label{st33move}
        \end{figure} 
          \begin{figure}[h]
  \centering
    \includegraphics[width=7cm,height=5cm]{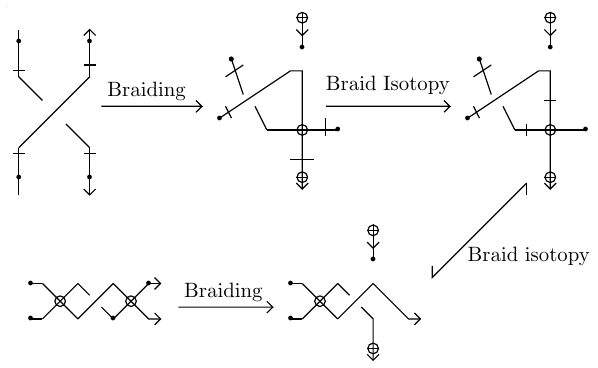}
        \caption{Case IV of T3 move}
        \label{st34move}
        \end{figure} 

\end{proof}
\begin{corollary}
     Two twisted virtual braids have isotopic closures if and only if they are twisted $L_v$-equivalent.
\end{corollary}
\subsection{Algebraic Markov type theorem for singular twisted virtual braids}
Consider a singular twisted virtual braid $b$ of degree $n$. Define a monomorphism $\iota_s^t: STVB_n \to STVB_{n+s+t}$, where $s$ and $t$ are non-negative integers such that $\iota_s^t(b)$ is a singular twisted virtual braid of degree $n + s + t$ obtained from $b$ by adding $s$ trivial strands to the left and $t$ trivial strands to the right.
\begin{theorem}\label{algebraic}
    Two singular twisted virtual braids have isotopic closures if and only if they differ by a finite sequence of braid relations in $STVB_\infty$ together with the following moves.
    \begin{itemize}
        \item[(i)] Real conjugation: $\sigma_i^{\pm 1} \alpha \sim \alpha \sigma_i^{\pm 1}$; virtual conjugation: $v_i \alpha \sim \alpha v_i$;\\
        singular conjugation: $\tau_i \alpha \sim \alpha \tau_i$; and twist conjugation: $\gamma_i \alpha \sim \alpha \gamma_i$,
        \item[(ii)] Right real and right virtual stabilization: $\iota^1_0(\alpha) \sigma_n^{\pm 1} \sim \alpha \sim \iota^1_0(\alpha) v_n$,
        \item[(iii)] Right algebraic under-threading: $\alpha \sim \iota^1_0(\alpha) \sigma_n^{\pm 1}v_{n-1}\sigma_n^{\mp 1}$; and \\
        left algebraic under-threading: $\alpha \sim \iota^0_1(\alpha)\sigma_1^{\pm 1}v_2\sigma_1^{\mp 1}$, 
        \item[(iv)] Right algebraic rs-threading: $\iota^1_0(\alpha) \tau_n v_{n-1}\sigma_n^{\pm 1} \sim \iota^1_0(\alpha) \sigma_n^{\pm 1}v_{n-1}\tau_n$; and \\
        left algebraic rs-threading: $\iota^0_1(\alpha) \tau_1 v_{2}\sigma_1^{\pm 1} \sim \iota^0_1(\alpha) \sigma_1^{\pm 1}v_{2}\tau_1$.
    \end{itemize}
\end{theorem}
    To prove this theorem it is sufficient to show that each algebraic Markov move can be obtained from a finite sequence of $L_v$-moves, which has already been proved in~\cite{cpas} except for the twist conjugation. It is easy to observe that the twist conjugation is a part of the $L_v$-moves.

\begin{corollary}
    Two twisted virtual braids have isotopic closures if and only if they differ by a finite sequence of braid relations in $TVB_\infty$ together with the following moves.
    \begin{itemize}
        \item[(i)] Real conjugation : $\sigma_i^{\pm 1} \alpha \sim \alpha \sigma_i^{\pm 1}$; virtual conjugation: $v_i \alpha \sim \alpha v_i$; and \\ 
 twist conjugation: $\gamma_i \alpha \sim \alpha \gamma_i$,
        \item[(ii)] Right real and right virtual stabilization: $\iota^1_0(\alpha) \sigma_n^{\pm 1}\sim \alpha \sim \iota^1_0(\alpha)v_n $,
        \item[(iii)] Right algebraic under-threading: $\alpha \sim \iota^1_0(\alpha) \sigma_n^{\mp 1}v_{n-1}\sigma_n^{\pm 1}$ and \\
        left algebraic under-threading: $\alpha \sim \iota^0_1(\alpha)\sigma_1^{\pm 1}v_2\sigma_1^{\mp 1}$. 
    \end{itemize}    
\end{corollary}
In the following subsection, we have demonstrated the potential to minimize the number of $L_v$-moves and algebraic moves in the context of singular twisted virtual braids.
\subsection{On threading moves of singular twisted virtual braids}
It turns out that if two singular twisted virtual braids are related by a left algebraic under-threading move (left algebraic $rs$-threading move) then they are related by a sequence of conjugation and a right algebraic under threading move (right algebraic $rs$-threading move). Similarly, if two singular twisted virtual braids are related by a LST2-move (LST3-move) then they are related by a ST4-move and a RST2-move (RST3-move). Thus we can remove left algebraic under threading move and left algebraic $rs$-threading move from the definition of Markov equivalence and also we can remove LST2-move and LST3-move from the definition of $L_v$ equivalence. Details are provided in the following theorems.
\begin{theorem}\label{LRST}
    If two twisted virtual braids of degree $n$ are related by a left algebraic under-threading move (left algebraic $rs$-threading move), then they are related by a sequence of conjugation and a right algebraic under-threading move (right algebraic $rs$-threading move).
\end{theorem}
\begin{proof}
    
Consider a monoid isomorphism $f_n: STVB_n \to STVB_n$ determined by 
\begin{align*}
\sigma_i & \mapsto \sigma_{n-i}, & \text{for } & i=1, \dots, n-1 \\ 
v_i & \mapsto v_{n-i}, & \text{for } & i=1, \dots, n-1 \\ 
\gamma_i & \mapsto \gamma_{n-i+1}, & \text{for } & i=1, \dots, n\\
\tau_i & \mapsto \tau_{n-i}, & \text{for } & i=1, \dots, n-1 . 
\end{align*}

Where, $f_n(b)$ a singular twisted virtual braid diagram obtained from the diagram $b$ by applying the above correspondence to the expression in terms of the generators. 

Let $b^* \in STVB_n$ such that
\begin{align*}
b^*  =  \prod_{i=1}^{n-1} (v_i v_{i-1} \dots v_1) \prod_{j=1}^{n} \gamma_j. 
\end{align*}

Define $F_n: STVB_n \to STVB_n$ such that
\begin{align*}
b & \mapsto b^* b {b^*}^{-1}  & \text{for } & b \in STVB_n. 
\end{align*}
It turns out to be a monoid isomorphism. It is easily seen that ${b^*}^2 = e$ and $F_n(b) = f_n(b)$ for $b \in STVB_n$.  
In particular $b$ and $f_n(b)$ are related by conjugation.

It is observable that the map $f_n$, facilitating the transition from the left diagram to the right one (or vice-versa), is indeed a conjugation map $F_n$. We use this observation to support our assertion. Further elaboration is provided below:

Let $\alpha$ and $b_2$ be singular twisted virtual braid diagrams related by a left algebraic under-threading move. 
Suppose that 
$$ b_1= \iota^0_1(\alpha) \quad \mbox{and} \quad 
b_2=    \alpha \sigma_1 v_2 \sigma_1^{-1}. $$  
Then 
$$ f_{n+1}(b_1) = \iota^1_0(f_{n}(\alpha))
\quad \mbox{and} \quad 
f_{n+1}(b_2) = f_{n}(\alpha) \sigma_n v_{n-1} \sigma_n^{-1}, $$ 
and hence $f_{n}(\alpha)$ and $f_{n+1}(b_2)$ are related by a right algebraic under threading move.   
Since $b_1$ is conjugate to $f_n(b_1)$ as elements of $STVB_n$ it implies $\alpha$ is conjugate to $f_{n}(\alpha)$. Also,
$b_2$ is conjugate to $f_{n+1}(b_2)$. Hence, $\alpha$ and $b_2$ are related by a sequence of conjugation and right algebraic under-threading move. 

Now, suppose $b_3$ and $b_4$ be singular twisted virtual braid diagrams related by a left algebraic $rs$-threading move. 
Suppose that
$$ b_3= \iota^0_1(\alpha) \tau_1 v_{2}\sigma_1 \quad \mbox{and} \quad 
b_4= \iota^0_1(\alpha) \sigma_1v_{2}\tau_1. $$  
Then 
$$ f_{n+1}(b_3) = \iota^1_0(f_{n}(\alpha))\tau_n v_{n-1}\sigma_n
\quad \mbox{and} \quad 
f_{n+1}(b_4) = \iota^1_0(f_{n}(\alpha) )\sigma_n v_{n-1} \tau_n, $$ 
and hence $f_{n+1}(b_3)$ and $f_{n+1}(b_4)$ are related by a right algebraic $rs$-threading move.   
Therefore, $b_3$ and $b_4$ are related by a sequence of conjugation and right algebraic  $rs$-threading move. 
\end{proof}
\begin{theorem}
    Two singular twisted virtual braids have isotopic closures if and only if they differ by a finite sequence of braid relations in $STVB_\infty$ together with the following moves.
    \begin{itemize}
        \item Real conjugation: $\sigma_i^{\pm 1} \alpha \sim \alpha \sigma_i^{\pm 1}$; virtual conjugation: $v_i \alpha \sim \alpha v_i$; \\
        singular conjugation: $\tau_i \alpha \sim \alpha \tau_i$; and twist conjugation: $\gamma_i \alpha \sim \alpha \gamma_i$,
        \item Right real and right virtual stabilzation: $\alpha \sigma_n^{\pm 1} \sim \alpha \sim \alpha v_n$,
        \item Right algebraic under-threading: $\alpha \sim \alpha \sigma_n^{\mp 1}v_{n-1}\sigma_n^{\pm 1}$,
        \item Right algebraic rs-threading: $\alpha \tau_n v_{n-1}\sigma_n^{\pm 1} \sim \alpha \sigma_n^{\pm 1}v_{n-1}\tau_n$.
    \end{itemize}
\end{theorem}
\begin{proof}
Use Theorem~\ref{algebraic} and Theorem~\ref{LRST}.
\end{proof}
\begin{corollary}
    Two twisted virtual braids have isotopic closures if and only if they differ by a finite sequence of braid relations in $TVB_\infty$ together with the following moves.
    \begin{itemize}
        \item[(i)] Real conjugation:  $\sigma_i^{\pm 1} \alpha \sim \alpha \sigma_i^{\pm 1}$; virtual conjugation: $v_i \alpha \sim \alpha v_i$; and twist conjugation: $\gamma_i \alpha \sim \alpha \gamma_i$,
        \item[(ii)] Right real and right virtual stabilization: $\alpha \sigma_n^{\pm 1} \sim \alpha \sim \alpha v_n $,
        \item[(iii)] Right algebraic under-threading: $\alpha \sim \alpha \sigma_n^{\mp 1}v_{n-1}\sigma_n^{\pm 1}$.
    \end{itemize}    
\end{corollary}
\begin{theorem}\label{1LRST}
    If two singular twisted virtual braids of degree $n$ are related by a LST2-move (LST3-move), then they are related by a sequence of ST4 and RST2-move (RST3-move).
\end{theorem}

\begin{proof}
In Figure~\ref{lutmove}, pictorially we have shown that LST2-move can be achieved by RST2-move and ST4-move. Similarly, the LST3-move can be obtained by the RST3-move and ST4-move as shown in Figure~\ref{lrstmove}.
     \begin{figure}[h]
  \centering
    \includegraphics[width=12cm,height=7cm]{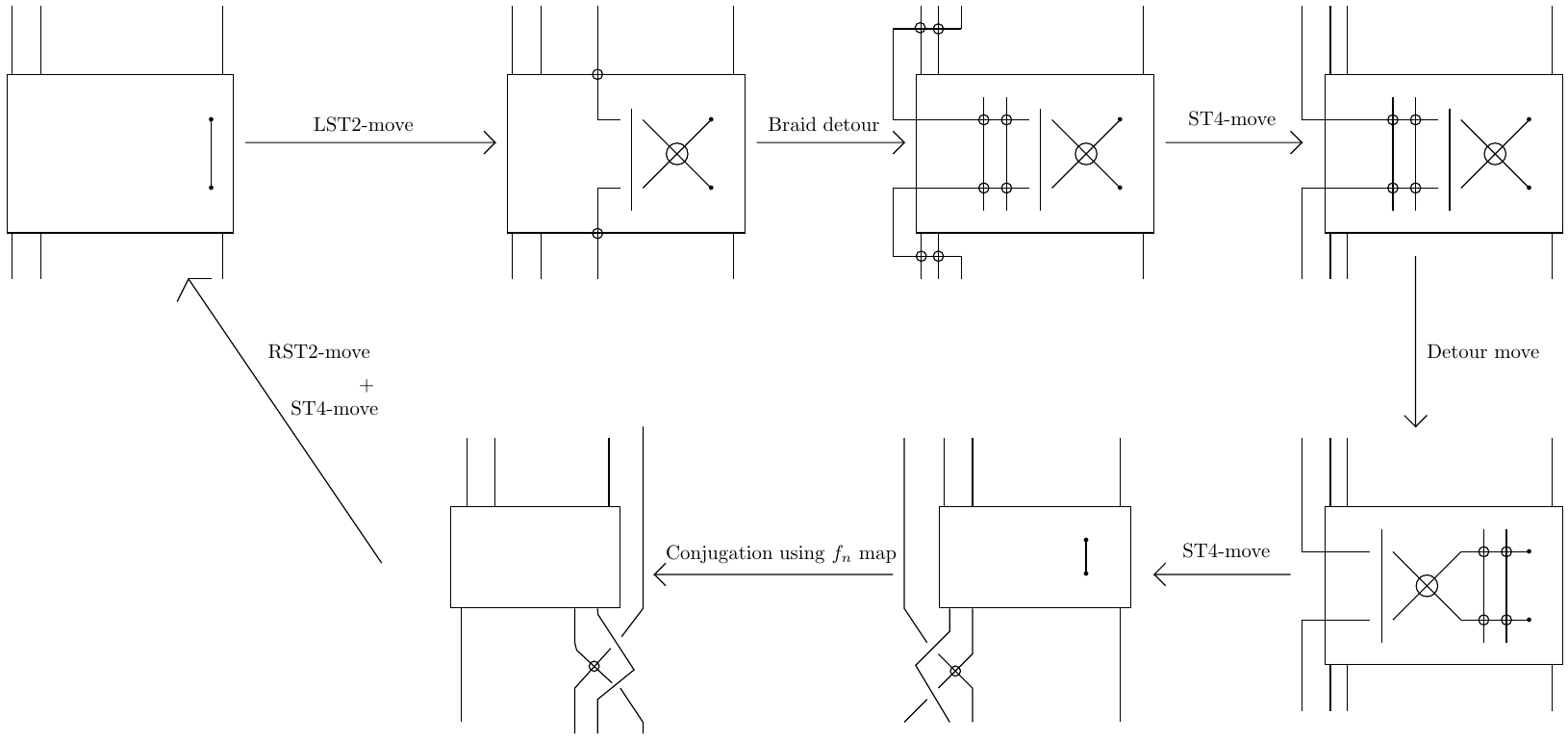}
        \caption{LST2-move using RST2-move and ST4-move}
        \label{lutmove}
        \end{figure} 
         \begin{figure}[H]
  \centering
    \includegraphics[width=11cm,height=7cm]{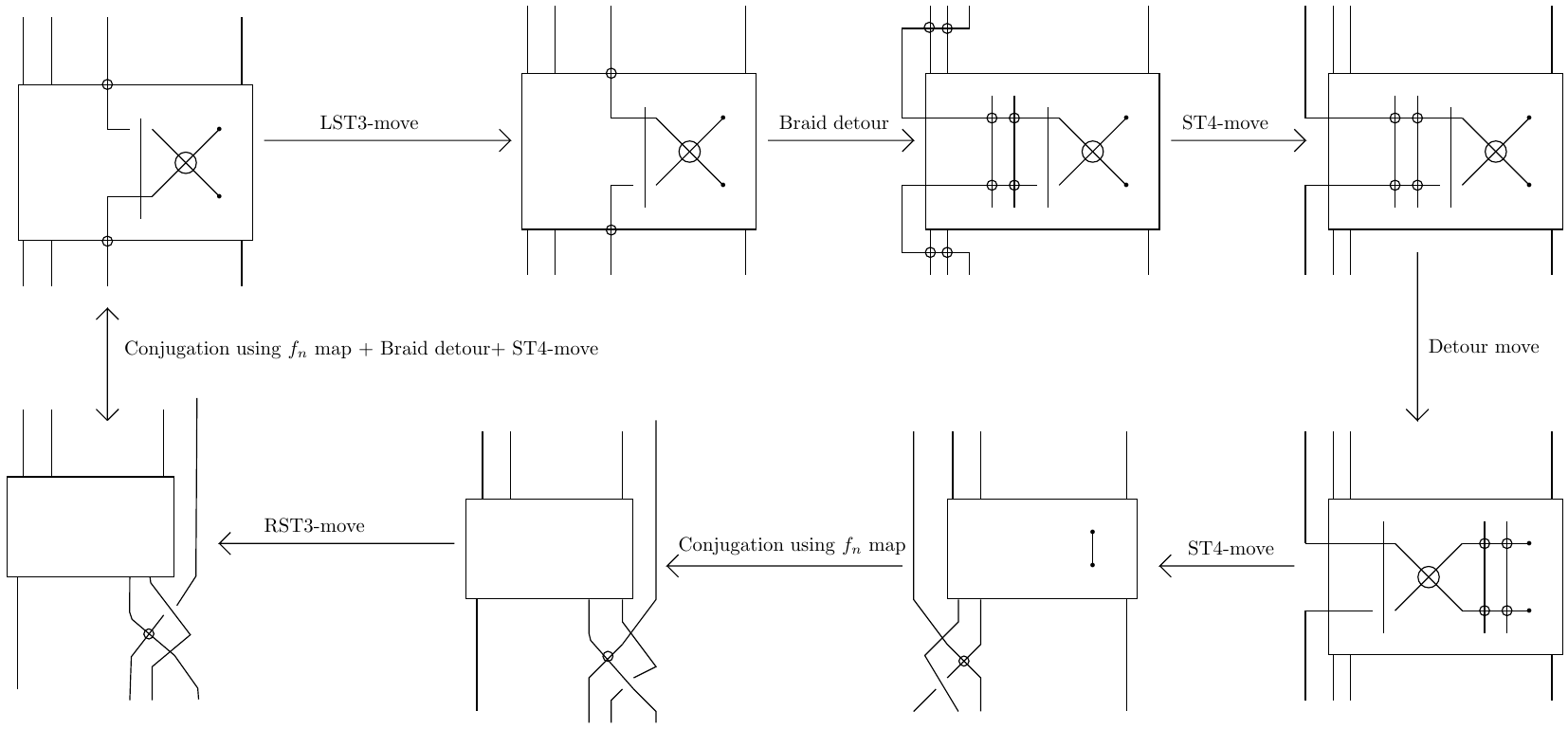}
        \caption{LST3-move using RST3-move and ST4-move}
        \label{lrstmove}
        \end{figure} 
\end{proof}
\begin{theorem}
    Two singular twisted virtual braids have isotopic closures if and only if they differ by sequence of ST1, RST2, RST3 and ST4 moves.
\end{theorem}
\begin{proof}
    Use Definition~\ref{def1} and Theorem~\ref{1LRST}.
\end{proof}
\begin{corollary}
    Two twisted virtual braids have isotopic closures if and only if they differ by sequence of ST1, RST2, RST4, and TST4 moves.
\end{corollary}

\section{Reduced presentation for singular twisted virtual braid monoid}
In this section, we show that the presentation of the singular twisted virtual braid monoid $STVB_n$ given in Theorem~\ref{thm:StandardPresentation2} can be reduced to a presentation with less generators and less relations by rewriting $\sigma_i$ $(i=2,\ldots, n-1)$, $\tau_i$ $(i=2,\ldots, n-1)$, and $\gamma_i$ $(i=2,\ldots, n)$ in terms of $\sigma_1$, $\tau_1$, $\gamma_1$ and $v_1, \dots, v_{n-1}$  as follows:
\begin{align}
    \sigma_i & =(v_{i-1}\ldots v_1)(v_i \ldots v_2)\sigma_1(v_2  \ldots v_i)(v_1  \ldots v_{i-1}) & \text{ for } & i=2,\ldots, n-1,  \label{1st reduction} \\ 
     \tau_i & =(v_{i-1}\ldots v_1)(v_i \ldots v_2)\tau_1(v_2  \ldots v_i)(v_1  \ldots v_{i-1}) & \text{ for } & i=2,\ldots, n-1,  \label{2nd reduction} \\
    \gamma_i & =(v_{i-1}\ldots v_1)\gamma_1(v_1  \ldots v_{i-1}) & \text{ for } & i=2,\ldots, n.  \label{3rd reduction}
\end{align}
These can be seen geometrically from their diagrams shown in Figure~\ref{o}.
\begin{figure}[h]
\centering
   \includegraphics[width=11cm,height=5cm]{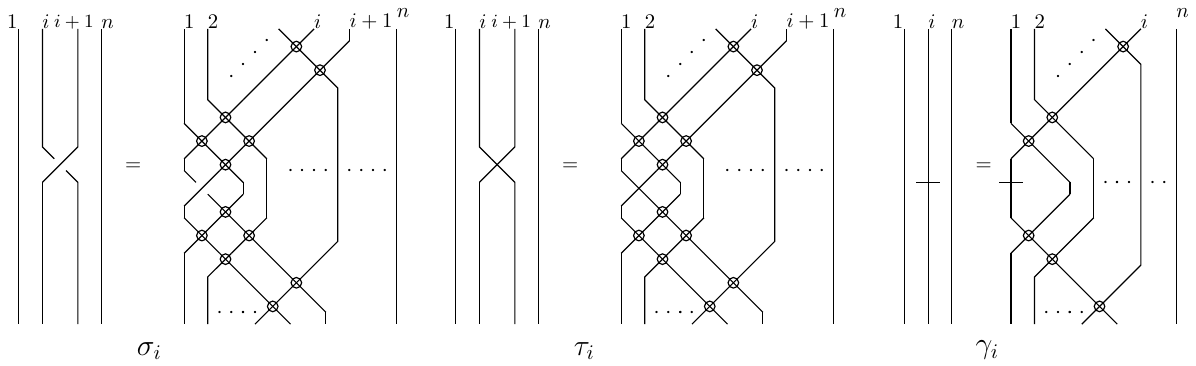}
      \caption{$\sigma_i$, $\tau_i$, and $\gamma_i$}
    \label{o}
      \end{figure}
        
Many relations are  already reduced in~\cite{cpas},~\cite{NPK},and~\cite{KS}.

\begin{theorem}\label{thm:ReducedPresentation}
The monoid $STVB_n$ has a presentation whose generators are $\sigma^{\pm 1}_1,\tau_1, \gamma_1, v_1,\dots, v_{n-1}$   
and the defining relations are as follows:  

\begin{align}
 v_i^2  & = e   & \text{ for }  & i=1,\ldots, n-1;     \label{relB-inverse-v}\\
 v_iv_j  & = v_jv_i  & \text{ for } & |i-j| > 1 ;  \label{relB-height-vv}\\
 v_iv_{i+1}v_i & = v_{i+1}v_iv_{i+1} & \text{ for } & i=1,\ldots, n-2; \label{relB-vvv}\\
 \sigma_1(v_2v_3v_1v_2\sigma_1v_2v_1v_3v_2) & = (v_2v_3v_1v_2\sigma_1v_2v_1v_3v_2)\sigma_1, &  & \label{relB-height-ss} \\
  (v_1\sigma_1v_1)(v_2\sigma_{1}v_2)(v_1\sigma_1v_1) & = (v_2\sigma_1v_2)(v_1\sigma_{1}v_1)(v_2\sigma_1v_2), & & \label{relB-sss} \\
 \sigma_1v_j  & = v_j\sigma_1 & \text{ for } & j = 3, \ldots, n-1; \label{relB-height-sv}\\
 \gamma_1^2  & = e,  & &  \label{relB-inverse-b} \\
 \gamma_1v_j & =  v_j\gamma_1 & \text{ for } & j = 2, \ldots, n-1; & \label{relB-height-bv}\\
  \gamma_1v_1\gamma_1v_1 & =v_1\gamma_1v_1\gamma_1, & &  \label{relB-height-bb}\\
  \gamma_1v_1v_2\sigma_1v_2v_1 & = v_1v_2\sigma_1v_2v_1\gamma_1, & & \label{relB-height-sb}\\
  \gamma_{1}v_1\gamma_1\sigma_{1} \gamma_1v_1\gamma_{1} & = \sigma_1. & & \label{relB-bv}\\
  \sigma_1\sigma^{-1}_1 & =e &  \label{sigma}\\
  \sigma_1\tau_1 &=\tau_1\sigma_1 & \label{tausigma}\\
  \tau_1 v_i & =v_i \tau_1 & \text{ for }  & i=3, \ldots n-1; \label{vtau}\\
  \tau(v_1v_2\sigma_1v_2v_1)\sigma_1 & = (v_1v_2\sigma_1v_2v_1)\sigma_1(v_1v_2\tau_1v_2v_1) \label{1}\\
  \tau_1(v_2v_3v_1v_2\sigma_1v_2v_1v_3v_2) &=(v_2v_3v_1v_2\sigma_1v_2v_1v_3v_2)\tau_1 \label{2}\\
  \tau_1(v_2v_3v_1v_2\tau_1v_2v_1v_3v_2) &=(v_2v_3v_1v_2\tau_1v_2v_1v_3v_2)\tau_1 \label{2}\\
   \gamma_1v_1v_2\tau_1v_2v_1 & = v_1v_2\tau_1v_2v_1\gamma_1, & & \label{relB-height-sb2}\\
  \gamma_{1}v_1\gamma_1\tau_{1} \gamma_1v_1\gamma_{1} & = \tau_1. & & \label{relB-bv2}
  \end{align}
  
\end{theorem}
The following lemmas are sufficient for the proof of the Theorem~\ref{thm:ReducedPresentation}.

\begin{lemma}~\label{p1}
    The braid relation $\tau_i\gamma_j  = \gamma_j\tau_i$ follow from the defining relations (\ref{2nd reduction}), and (\ref{3rd reduction}), the other relations, and the reduced relations $ \gamma_1v_1v_2\tau_1v_2v_1  = v_1v_2\tau_1v_2v_1\gamma_1$ and $\gamma_1v_j=v_j\gamma_1, \text{ for } j\neq 1$.
\end{lemma}
\begin{proof}
We can see that the relation $(\ref{2nd reduction})$ is equivalent to the relation 

$\tau_{i+1} = v_i v_{i+1} \tau_i v_{i+1} v_i$, $\forall 1\leq i\leq n-1$.  

First we show $(\ref{1rel-height-sb})$ when  $j=1$, i.e., 
$\tau_i\gamma_1=\gamma_1\tau_i$ for $i \neq 1$.  
We apply induction on $i$, with initial condition $i=2$.
Using the relation $(\ref{2nd reduction})$ and $(\ref{relB-height-sb2})$, the following relation follows: 
\begin{equation}
  \tau_2\gamma_1=\gamma_1\tau_2  \label{i=1}
\end{equation}
Assuming $\tau_i\gamma_1=\gamma_1\tau_i$, we obtain $\tau_{i+1}\gamma_1=\gamma_1\tau_{i+1}$ as follows: 
\begin{align*}
    \tau_{i+1}\gamma_1 & = v_{i}v_{i+1}\tau_{i}v_{i+1}\underline{v_{i}\gamma_1}~~~~~~~~~~~~~~~~~~~~~ \text{ (by (\ref{relB-height-bv}))}\\
                    & = v_{i}v_{i+1}\tau_{i}\underline{v_{i+1}\gamma_1}v_{i}~~~~~~~~~~~~~~~~~~~~~ \text{ (by (\ref{relB-height-bv}))}\\
                    & = v_{i}v_{i+1}\underline{\tau_{i}\gamma_1}v_{i+1}v_{i}~~~~~~~~~~~~~~~~~~~~~\text{ (by (\ref{i=1}))}\\
                    & = v_{i}\underline{v_{i+1}\gamma_1}\tau_{i}v_{i+1}v_{i}~~~~~~~~~~~~~~~~~~~~~\text{ (by (\ref{relB-height-bv}))}\\
                    & = \underline{v_{i}\gamma_1}v_{i+1}\tau_{i}v_{i+1}v_{i}~~~~~~~~~~~~~~~~~~~~~ \text{ (by (\ref{relB-height-bv}))}\\
                    & = \gamma_1 \underline{v_{i}v_{i+1}\tau_{i}v_{i+1}v_{i}}~~~~~~~~~~~~~~~~~~~~~ \text{ (by (\ref{2nd reduction}))}\\
                    & = \gamma_1\tau_{i+1}.
\end{align*}
Hence,\begin{equation}
    \tau_i\gamma_1=\gamma_1\tau_i \quad \text{ for } i \neq 1. \label{4th reduction}
\end{equation}

Now, we show relation $(\ref{1rel-height-sb})$: $\tau_i\gamma_j = \gamma_j\tau_i$ for $j\neq i, i+1$.

Case(i) Suppose $j\leq i-1$. Then 
\begin{align*}
    \tau_i\gamma_j & = \underline{\tau_i }(v_{j-1}\ldots v_1)\gamma_1(v_1 \ldots v_{j-1})~~~~~~~~~~~ \text{ (by (\ref{rel-height-vv1}))}\\
                & = (v_{j-1}\ldots v_1)\underline{\tau_i \gamma_1}(v_1 \ldots v_{j-1}) ~~~~~~~~~~~ \text{ (by (\ref{4th reduction}))}\\
                & = (v_{j-1}\ldots v_1)\gamma_1\underline{\tau_i }(v_1 \ldots v_{j-1})~~~~~~~~~~~ \text{ (by (\ref{rel-height-vv1}))}\\
                & = \underline{(v_{j-1}\ldots v_1)\gamma_1(v_1 \ldots v_{j-1})}\tau_i~~~~~~~~~~~ \text{ (by (\ref{3rd reduction}))}\\
                & = \gamma_j\tau_i.
\end{align*}

Case(ii) Suppose $j\geq i+2$. Then 
\begin{align*}
    \tau_i\gamma_j & = \underline{\tau_i }(v_{j-1}\ldots v_1)\gamma_1(v_1 \ldots v_{j-1})~~~~~~~~~~~~~~~~~~~~~~~~~~~~~~~~~~~~~~~~~~~~~~~ \text{ (by (\ref{rel-height-vv1}))}\\
                & = (v_{j-1}\ldots v_{i+2}) \underline{\tau_i} (v_{i+1}\ldots v_1)\gamma_1(v_1 \ldots v_{j-1})~~~~~~~~~~~~~~~~~~~~~~~~~~~~~ \text{ (by (\ref{2nd reduction}))}\\ 
                & = (v_{j-1}\ldots v_{i+2}) v_{i+1} v_i \tau_{i+1} \underline{v_i v_{i+1}} (\underline{v_{i+1} v_i} v_{i-1} \ldots v_1)\gamma_1(v_1 \ldots v_{j-1})~~ \text{ (by (\ref{relB-inverse-v}))}\\                 
                & = (v_{j-1}\ldots v_{i})\underline{\tau_{i+1} }(v_{i-1}\ldots v_1)\gamma_1(v_1 \ldots v_{j-1})~~~~~~~~~~~~~~~~~~~~~~~~~~~~~ \text{ (by (\ref{rel-height-vv1}))}\\
                & = (v_{j-1}\ldots v_{i})(v_{i-1}\ldots v_1)\underline{\tau_{i+1}\gamma_1 }(v_1 \ldots v_{j-1})~~~~~~~~~~~~~~~~~~~~~~~~~~~~~\text{ (by (\ref{4th reduction}))}\\
                & = (v_{j-1}\ldots v_{1}) \gamma_1 \underline{\tau_{i+1}} (v_1 \ldots v_{j-1})~~~~~~~~~~~~~~~~~~~~~~~~~~~~~~~~~~~~~~~~~~~~ \text{ (by (\ref{rel-height-vv1}))}\\
                & = (v_{j-1}\ldots v_{1}) \gamma_1 (v_1 \ldots v_{i-1}) \underline{\tau_{i+1}} (v_i \ldots v_{j-1})~~~~~~~~~~~~~~~~~~~~~~~~~~~~~\text{ (by (\ref{2nd reduction}))}\\
                & = (v_{j-1}\ldots v_{1}) \gamma_1 (v_1 \ldots v_{i-1}) v_i v_{i+1} \tau_i \underline{v_{i+1} v_i} ( \underline{v_i v_{i+1}} v_{i+2} \ldots v_{j-1})~~~~~\text{ (by (\ref{relB-inverse-v}))}\\
                & = (v_{j-1}\ldots v_{1}) \gamma_1 (v_1 \ldots v_{i+1}) \underline{\tau_{i}} (v_{i+2} \ldots v_{j-1})~~~~~~~~~~~~~~~~~~~~~~~~~~~~~\text{ (by (\ref{2nd reduction}))}\\
                & = \underline{(v_{j-1}\ldots v_{1}) \gamma_1 (v_1 \ldots v_{i+1})(v_{i+2} \ldots v_{j-1})}\tau_{i} ~~~~~~~~~~~~~~~~~~~~~~~~~~~~~\text{ (by (\ref{3rd reduction}))}\\
                & = \gamma_j\tau_i.
\end{align*}
\end{proof}

\begin{lemma}\label{p2}
    The braid relation $v_{i} \sigma_i v_{i}  = \gamma_{i+1} \gamma_i \tau_{i} \gamma_i \gamma_{i+1}$ follow from the defining relations (\ref{2nd reduction}) and (\ref{3rd reduction}), the virtual relations, and the reduced relations $ \gamma_1v_1\gamma_1\tau_1\gamma_1v_1\gamma_1 = \tau_1$ and $\gamma_1v_j=v_j\gamma_1, \text{ for } j\neq 1$.
\end{lemma}
\begin{proof} For the proof we will use the following relation:
\begin{equation}
    (v_1 \ldots v_{i-1})v_i(v_{i-1}\ldots v_1)=(v_i\ldots v_2)v_1(v_2\ldots v_i) \label{v}
\end{equation}
\begin{align*}
    \gamma_{i+1}\gamma_{i}\tau_i\gamma_{i}\gamma_{i+1} 
 & = (v_i\ldots v_1)\gamma_1(v_1\ldots v_i)(v_{i-1} \ldots v_1)\gamma_1 \underline{(v_1 \ldots v_{i-1})(v_{i-1} \ldots v_1)}(v_{i} \ldots v_2)\\
    & \tau_1(v_2 \ldots v_{i}) \underline{(v_1 \ldots v_{i-1})(v_{i-1} \ldots v_1)} \gamma_1(v_1 \ldots v_{i-1})(v_i \ldots v_1)\gamma_1(v_1 \ldots v_i)~~~~\text{ (by (\ref{relB-inverse-v}))}\\
 & = (v_i\ldots v_1)\gamma_1\underline{(v_1\ldots v_{i-1} v_i v_{i-1} \ldots v_1)} \gamma_1(v_{i} \ldots v_2)\tau_1(v_2 \ldots v_{i})\gamma_1\\
     & \underline{(v_1 \ldots v_{i-1} v_i v_{i-1} \ldots v_1)}\gamma_1(v_1 \ldots v_i)~~~~~~~~~~~~~~~~~~~~~~~~~~~~~~~~~~~~~~~~~~~~~~~~~\text{ (by (\ref{v}))}\\
 & = (v_i\ldots v_1)\underline{\gamma_1}(v_i\ldots v_2 v_1 v_2 \ldots v_i)\underline{\gamma_1}(v_{i} \ldots v_2)\tau_1(v_2 \ldots v_{i}) \underline{\gamma_1}\\
     & (v_i \ldots v_2 v_1 v_2 \ldots v_i) \underline{\gamma_1}(v_1 \ldots v_i)~~~~~~~~~~~~~~~~~~~~~~~~~~~~~~~~~~~~~~~~~~~~~~~~~~~~~~~\text{ (by (\ref{relB-height-bv}))}\\
     & = (v_i\ldots v_1)(v_i\ldots v_{2})\gamma_1 v_1\underline{(v_{2} \ldots v_i)(v_{i} \ldots v_2)}\gamma_1\tau_1\gamma_1\underline{(v_2 \ldots v_{i})(v_i \ldots v_{2})}\\
     & v_1 \gamma_1 (v_{2} \ldots v_i)(v_1 \ldots v_i)~~~~~~~~~~~~~~~~~~~~~~~~~~~~~~~~~~~~~~~~~~~~~~~~~~~~~~~~~~~~~~~~\text{ (by (\ref{relB-inverse-v}))}\\
     & = (v_i\ldots v_1)(v_i\ldots v_{2})\underline{\gamma_1 v_1\gamma_1\tau_1\gamma_1v_1\gamma_1}(v_{2} \ldots v_i)(v_1 \ldots v_i)~~~~~~~~~~~~~~~~~~~~~~\text{ (by (\ref{relB-height-sb2}))}\\
     & = v_i\underline{(v_{i-1}\ldots v_1)(v_i\ldots v_{2})\tau_1(v_{2} \ldots v_i)(v_1 \ldots v_{i-1})}v_i~~~~~~~~~~~~~~~~~~~~~~~~~~~~\text{ (by (\ref{2nd reduction}))}\\
     & = v_i\tau_iv_i.
\end{align*}

\end{proof}

%%%%%%%%%%%%%%%%%%%%%%%%%%%%%%%%%%%%%%%%%%%%%%%%%%%%%%%%%%%%%%%
\section{Concluding remarks}

In this paper, we explore singular twisted virtual braids and the associated singular twisted virtual braid monoid. We establish theorems regarding singular twisted links such as the Alexander theorem and the Markov theorem. Additionally, we present a monoid presentation and a reduced monoid presentation for the singular twisted virtual braid monoid. As a potential direction for future research, developing invariants for both singular twisted virtual braids and singular twisted links would be of considerable interest.

%%%%%%%%%%%%%%%%%%%%%%%%%%%%%%%%%%%%%%%%%%%%%%%%%%%%%%%%%%%%%%%%%%%%%%%%%%%%%%%%%%%%%%

\subsection*{Acknowledgements}

The first author would like to thank the University Grants Commission(UGC), India, for Research Fellowship with NTA Ref.No.191620008047. The second author acknowledges the support given by the NBHM, Government of India under grant-in-aid with F.No.02011/2/20223NBHM(R.P.)/R\&D II/970. This work was partially supported by the FIST program of the Department of Science and Technology, Government of India, Reference No. SR/FST/MS-I/2018/22(C).
%%%%%%%%%%%%%%%%%%%%%%%%%%%%%%%%%%%%%%%%%%%%%%%%%%%%%%%%%%%%%%%%%%%%%%%%%%%%%%%%%%%%%%%

\noindent Department of Mathematics, Indian Institute\\
 of Technology Ropar, Punjab, India.
\begin{verbatim} Email: komal.20maz0004@iitrpr.ac.in, prabhakar@iitrpr.ac.in  \end{verbatim}

\end{document}